\documentclass[12pt,reqno]{amsproc}
  \usepackage{latexsym} 
  \usepackage[all]{xy}
  \usepackage{amsfonts} 
  \usepackage{amsthm} 
  \usepackage{amsmath} 
  \usepackage{amssymb}
  \usepackage{pifont}  
  \usepackage{enumerate}
  \xyoption{2cell}

 \usepackage{tikz}
\usetikzlibrary{arrows}

  \def\su#1{{\sp{[#1]}}}

  \def\<{{\langle}} 
  \def\>{{\rangle}}

  \def\note#1{{}}

  \def\note#1{} 

     \def\cT{{\mathcal T}}

  \def\rhom#1#2#3{{{\rm Hom}\sb{#1}(#2,#3)}}

  \def\beq{\begin{equation}} 
  \def\eeq{\end{equation}}

  \def\id{\mathrm{id}}

  \def\ot{{\otimes}}

  \newcounter{zlist} 
  \newenvironment{zlist}{\begin{list}{(\arabic{zlist})}{ 
  \usecounter{zlist}\leftmargin2.5em\labelwidth2em\labelsep0.5em 
  \topsep0.6ex
  \parsep0.3ex plus0.2ex minus0.1ex}}{\end{list}}

  \newcounter{blist}

  \newcounter{rlist}

\def\stac#1{\raise-.2cm\hbox{$\stackrel{\displaystyle\otimes}{\scriptscriptstyle{#1}}$}}

\def\cten#1{\raise-.2cm\hbox{$\stackrel{\displaystyle\widehat{\otimes}}
{\scriptscriptstyle{#1}}$}}

  \def\Label#1{\label{#1}\ifmmode\llap{[#1] }\else 
  \marginpar{\smash{\hbox{\tiny [#1]}}}\fi} 
  \def\Label{\label}

  \newtheorem{proposition}{Proposition}[section]
  \newtheorem{lemma}[proposition]{Lemma}

\theoremstyle{definition} 

  \newtheorem{definition}[proposition]{Definition}
\newtheorem{example}[proposition]{Example}

  \theoremstyle{remark} 
  \newtheorem{remark}[proposition]{Remark}

  \newcounter{c} 
   
  \newcommand{\etyk}[1]{\vspace{-7.4mm}$$\begin{equation}\Label{#1} 
  \addtocounter{c}{1}} 
  \renewcommand{\]}{\ifnum \value{c}=1 $$\else \end{equation}\fi} 
\def\ot{\otimes}

\def\KK{{\mathbb K}}

\def\NN{{\mathbb N}}

\newcommand{\Cc}{\mathcal{C}}

\newcommand{\Tt}{\mathcal{T}}

\def\*C{{}^*\hspace*{-1pt}{\Cc}}

\def\text#1{{\rm {\rm #1}}}

 \def\1{\mathbf{1}}

\begin{document}

\title{Curved Rota-Baxter systems}

\author{Tomasz Brzezi\'nski}
 \address{Department of Mathematics, Swansea University,   Swansea SA2 8PP, U.K.}
 \address{ 
Department of Mathematics, University of Bia{\l}ystok, K.\ Cio{\l}kowskiego  1M,
15-245 Bia{\l}ystok, Poland}
  \email{T.Brzezinski@swansea.ac.uk}   
 
\begin{abstract}
Rota-Baxter systems are modified by the inclusion of a curvature term. It is shown that, subject to specific properties of the curvature form, curved Rota-Baxter systems $(A,R,S,\omega)$ induce associative and (left) pre-Lie products on the algebra $A$. It is also shown that if both Rota-Baxter operators coincide with each other and the curvature is $A$-bilinear, then the (modified by $R$) Hochschild cohomology ring over $A$ is a curved differential graded algebra.
\end{abstract}
\maketitle

\section{Introduction}

Rota-Baxter operators appeared in the work of Baxter \cite{Bax:ana} on differential operators on commutative Banach algebras, being particularly useful in relation to the Spitzer identity. The defining identity of a (weight zero) Rota-Baxter operator can be understood as encoding the integration by part law in the way analogous to that in which the Leibniz rule characterises differentiation. Rota and his school realised the usefulness of such operators in combinatorics in particular in the context of Warning's formula relating the power sum symmetric functions to  elementary symmetric functions \cite{Rot:Bax}. Aguiar connected Rota-Baxter operators with Yang-Baxter operators and, inspired by this connection, introduced infinitesimal bialgebras \cite{Agu:inf}. Furthermore, a relation of Rota-Baxter algebras to dendriform algebras of Loday \cite[Section~5]{Lod:dia} was explored in \cite{Ebr:Lod}, \cite{Agu:info}. Through their employment in combinatorics on one hand and connection to the Yang-Baxter equation on the other, Rota-Baxter algebras found their way into mathematical physics, in particular the renormalisation of quantum field theories \cite{EbrGuo:Rot} and, most recently, integrable systems \cite{Sza:cla}. For a short and accessible review of Rota-Baxter algebras the reader is referred to \cite{Guo:wha}.

In an attempt to develop and extend aformentioned connections between Rota-Baxter algebras, dendriform algebras and infinitesimal bialgebras the notion of a {\em Rota-Baxter system} has been introduced in \cite{Brz:Rot}.  In particular it has been shown that to any Rota-Baxter system one can associate a  dendriform algebra  and, in fact, any dendriform algebra of a particular kind arises from a Rota-Baxter system. Consequently, Rota-Baxter systems yield pre-Lie and associative algebra structures. Furthermore, in parallel to the relation between the Rota-Baxter identity and the integration by parts law, an example of a Rota-Baxter system termed a twisted Rota-Baxter operator, has been shown to satisfy the integration by parts law of the Jackson q-integral.

In this note we modify the definition of a Rota-Baxter system by including a curvature term and then derive the conditions that the curvature has to satisfy in order to yield a pre-Lie, associative or curved differential graded algebra structures.

This note deals with the properties of an algebraic system consisting of an algebra and three maps, which satisfy properties listed in the following

\begin{definition}\label{def.R-B}
A system $(A,R,S,\omega)$ consisting of an associative (but not necessarily unital) algebra $A$ over a commutative ring $\KK$ and $\KK$-linear maps $R,S:A\to A$, $\omega: A\ot A\to A$ is called a {\em curved Rota-Baxter system} if, for all $a,b\in A$,
\begin{equation} \label{r}
R(a)R(b) = R\left(R(a) b + a S(b)\right) + \omega(a\otimes b),
\end{equation}
\begin{equation} \label{s}
S(a)S(b) = S\left(R(a) b + a S(b)\right) + \omega(a\otimes b). 
\end{equation}
The maps  $R$ and $S$ are termed {\em Rota-Baxter operators} and $\omega$ is called a {\em curvature}.
\end{definition}

Curved Rota-Baxter systems generalise Rota-Baxter operators and algebras at least in a threefold way. First, when the curvature vanishes, the triple $(A,R,S)$ is a Rota-Baxter system and hence the choice of $S$ to be $R +\lambda\id$ with  $\lambda\in \KK$ or $R$ to be $S+\lambda\id$ makes it into a Rota-Baxter algebra of weight $\lambda$. On the other hand, a Rota-Baxter algebra of weight $\lambda$ is obtained from  $(A,R,S,\omega)$ by setting $R=S$ and $\omega(a\ot b) = \lambda R(ab)$, for all $a,b\in A$.


\section{Curved Rota-Baxter systems and associative  algebras} 
As explained in \cite{Brz:Rot}, a Rota-Baxter $(A,R,S)$ system yields an  associative product on $A$. This remains true for a curved Rota-Baxter system provided the curvature satisfies specific condition.

\begin{proposition}\label{prop.ass}
Let $(A,R,S,\omega)$ be a curved Rota-Baxter system and define, for all $a,b\in A$,
\begin{equation}\label{prod}
a *b = R(a) b + a S(b).
\end{equation}
Then $(A,*)$ is an associative algebra if and only if, for all $a,b,c\in A$,
\begin{equation}\label{ass}
a\omega(b\ot c) = \omega(a\ot b)c.
\end{equation}
In particular, if $A$ has an identity, then $(A,*)$ is an associative algebra if and only if there exists a central element $\kappa \in A$ such that, for all $a,b\in A$,
\begin{equation}\label{unital}
\omega(a\ot b) = \kappa ab.
\end{equation}
\end{proposition}

\begin{proof}
Note that, in terms of the product $*$, the curved Rota-Baxter conditions \eqref{r}--\eqref{s} can be equivalently written as 
$$
R(a)R(b) = R(a*b) + \omega(a\otimes b), \qquad S(a)S(b) = S(a* b) + \omega(a\otimes b).
$$
Therefore,
\begin{eqnarray*}
(a*b)*c - a*(b*c) &=& R(a*b)c + (a*b) S(c) - R(a) (b*c) - a S(b*c)\\
&=& R(a)R(b)c - \omega(a\otimes b) c + R(a)b S(c) + a S(b)S(c)\\
&& - R(a)R(b) c - R(a)b S(c) - aS(b)S(c) + a\omega(b\ot c)\\
&=& a\omega(b\ot c)- \omega(a\otimes b) c ,
\end{eqnarray*}
which yields the first assertion.

If $\omega$ is given by \eqref{unital} (with central $\kappa$), then it clearly satisfies condition \eqref{ass}. On the other hand, if $A$ is unital, then \eqref{ass} implies that, for all $a,b\in A$,
$$
\omega(a\ot b) = \omega(1\ot 1) ab,
$$
so that the curvature is fully determined by $\kappa:= \omega(1\ot 1)$. Again by the repeated use of \eqref{ass} one finds 
$$
a\kappa = a\omega (1\ot 1) = \omega(a\ot 1) = \omega(1\ot a) = \omega(1\ot 1) a = \kappa a,
$$
i.e.\ $\kappa$ is a central element and the stated form \eqref{unital} of the curvature is thus obtained.
\end{proof}

In a way similar to \cite[Lemma~2.9]{Brz:Rot}, one can analyse the associativity of the product \eqref{prod} from the perspective of {\em weak pseudotwistors} \cite{PanVan:twi}. The latter notion needs to be modified by introducing of curvature.

\begin{definition}\label{def.twistor}
Let $A$ be an algebra with associative product $\mu:A\ot A\to A$. A $\KK$-linear map $T:A\ot A\to A\ot A$ is called a {\em curved 
 weak pseudotwistor} if there exist  $\KK$-linear maps $\Tt:A\ot A\ot A\to A\ot A\ot A$ and $\omega:A\ot A\to A$, rendering commutative the following diagrams:
\begin{equation}\label{bow-tie}
\xymatrix{
A\!\ot\!  A\!\ot\!  A 
\ar[rrr]^-{\id \otimes \mu} &&& A\!\ot\!  A 
\ar[d]^{T}  &&& A\!\ot\!  A\!\ot\!  A \ar[lll]_-{\mu \ot \id} \\
A\!\ot\! A\!\ot\!  A \ar[u]^{\id \ot T} \ar[rrr]_{(\id \ot \mu)\circ \Tt - \id\ot \omega} & & &  A\!\ot\!  A  & & &  A\!\ot\!  A\!\ot\!  A \ar[u]_{T\ot \id}
\ar[lll]^{(\mu\ot \id)\circ \Tt - \omega\ot \id}} 
\end{equation}
\begin{equation}\label{omega}
\xymatrix{A\!\ot\!  A\!\ot\!  A 
\ar[rr]^-{\id \otimes \omega}\ar[d]_{\omega\ot \id} && A\!\ot\!  A 
\ar[d]^{\mu}\\
A\!\ot\!  A\ar[rr]_\mu && A\ .} 
\end{equation}
The map $\Tt$ is called a {\em weak companion} of $T$ and $\omega$ is called the {\em curvature} of $T$.
\end{definition}
\begin{lemma}\label{lemma.twistor}
Let $T:A\ot A\to A\ot A$ be a {\em curved 
 weak pseudotwistor} with companion $\cT$ and curvature $\omega$. Then $\mu\circ T$ is an associative product on $A$.
 \end{lemma}
 \begin{proof}
 With the help of both diagrams in Definition~\ref{def.twistor} and associativity of $\mu$ one easily computes,
 \begin{eqnarray*}
 \mu\circ T \circ (\id \ot \mu\circ T) &=& \mu \circ (\id \ot \mu)\circ \cT - \mu\circ (\id \ot \omega) \\
 & =& \mu \circ (\mu \ot\id)\circ \cT - \mu \circ (\omega \ot\id)\\
 &=& \mu \circ T\circ (\mu\circ T\ot \id),
 \end{eqnarray*}
 as required.
 \end{proof}
 
 \begin{lemma}\label{lem.RBtwistor}
 Let $(A,R,S,\omega)$ be a curved Rota-Baxter system with the curvature that satisfies \eqref{ass}, and define, for all $a,b,c\in A$,
 $$
 T(a\ot b) = R(a)\ot b + a\ot S(b), 
$$
and
$$
\cT(a\ot b\ot c) = R(a)\ot R(b)\ot c + R(a)\ot b\ot S(c) + a\ot S(b)\ot S(c).
$$
Then $T$ is a curved weak pseudotwistor with the weak companion $\cT$ and curvature $\omega$.
 \end{lemma}
 \begin{proof}
 The commutativity of diagram \eqref{omega} is equivalent to \eqref{ass}. To check the commutativity of the left square in diagram \eqref{bow-tie}, let us take any $a,b,c \in A$ and compute
 \begin{eqnarray*}
 T\circ (\id\ot \mu\circ T)(a\ot b\ot c) &=& T(a\ot R(b)c + a\ot bS(c))\\
 &=& R(a)\ot (R(b)c + bS(c)) + a\ot S(R(b)c + bS(c))\\
 &=& R(a)\ot R(b)c + R(a)\ot  bS(c) + a\ot S(b)S(c)\\
 && - a\ot \omega(b\ot c)\\
 &=& ((\id \ot \mu)\circ \Tt - \id\ot \omega)(a\ot b\ot c),
 \end{eqnarray*}
 where the condition \eqref{s} has been used in the derivation of the third equality.
 The commutativity of the right square in diagram \eqref{bow-tie} is checked in a similar way.
 \end{proof}
 
 It is clear that Proposition~\ref{prop.ass} can be understood as a consequence of Lemma~\ref{lemma.twistor} and Lemma~\ref{lem.RBtwistor}.
 
\section{Curved Rota-Baxter systems and curved differential graded algebras}

A triple $(A,d,\omega)$ consisting of an $\NN$-graded algebra $A = \oplus_{n\in \NN} A^n$, a degree-one graded derivation $d$ of $A$, and $\omega \in A^2$ such that, for all $a\in A$, 
$$
d\circ d (a) = [\omega, a], \qquad d(\omega) =0,
$$
where $[-,-]$ denotes the (graded) commutator, is called a {\em curved differential graded algebra}; see \cite{GetJon:alg}, \cite{Pos:non}.

\begin{proposition}\label{prop.cdga}
 Let $(A,R,S,\omega)$ be a curved Rota-Baxter system with the curvature satisfying \eqref{ass}. For all $n\in \NN$, set 
$$
\Omega^n(A) = \rhom{\KK} {A^{\ot n}} A,
$$
and view $\Omega(A) = \oplus_{n\in \NN} \Omega^n(A)$ as a graded algebra via
$$
(f g) (a_1,\ldots, a_{m+n}) = f(a_1,\ldots , a_m)g(a_{m+1}, \ldots, a_{m+n}),
$$
for all $f\in \Omega^m(A)$ and $g\in \Omega^n(A)$. 
Define the map $d: \Omega^n(A) \to \Omega^{n+1}(A)$ by
\begin{eqnarray}
d(f)(a_0, a_1\ldots , a_n) &=& R(a_0)f(a_1,\ldots , a_n) + \sum_{k=1}^{n-1} (-1)^k f(a_0, \ldots , a_{k-1}*a_k, \ldots , a_n ) \nonumber \\
&&+ (-1)^n f(a_0,\ldots , a_{n-1})S(a_n), \label{d}
\end{eqnarray}
for all $f \in \Omega^n(A)$, where $*$ is the product defined by \eqref{prod}. Then:
\begin{zlist}
\item For all $f \in \Omega^n(A)$,
$$
d\circ d (f)  = [\omega , f].
$$
\item If $S=R$ and $\omega$ is an $A$-bimodule map, then $(\Omega(A), d, \omega)$ is a curved differential graded algebra. 
\end{zlist}
\end{proposition}

\begin{proof}
(1) Since the product $*$ is associative the repeated application of $d$ to $f \in \Omega^n(A)$ yields  cancellation of all terms that involve the $*$-product of arguments in $f$, and so one is left with
\begin{eqnarray*}
d\circ d (f)(a_0, a_1\ldots , a_{n+1}) &=& R(a_0)R(a_1) f(a_2,\ldots , a_{n+1}) - R(a_0*a_1) f(a_2,\ldots , a_{n+1}) \\
&&\hspace{-1cm} + f(a_0,\ldots , a_{n-1})S(a_n*a_{n+1}) - f(a_0,\ldots , a_{n-1})S(a_n)S(a_{n+1})\\
&=& (\omega f - f\omega)(a_0, a_1\ldots , a_{n+1}),
\end{eqnarray*}
by equations \eqref{r}--\eqref{s}.

(2) A straightforward calculation shows that if $R=S$, then $d$ is a graded derivation. Furthermore, since $\omega$ is an $A$-bimodule map and it satisfies \eqref{ass}, for all $a,b,c\in A$,
$$
\omega(ab\ot c) = \omega(a\ot bc).
$$
Hence,
\begin{eqnarray*}
d(\omega)(a,b,c) &=& R(a)\omega(b\ot c) - \omega(a*b\ot c) + \omega(a\ot b*c) - \omega(a\ot b) R(c)\\
&=& \omega(R(a)b\ot c) - \omega(R(a)b\ot c) - \omega(aR(b)\ot c) \\
&&+ \omega(a\ot R(b)c) + \omega(a\ot bR(c)) - \omega(a\ot bR(c)) =0.
\end{eqnarray*}
Therefore, $(\Omega(A), d,\omega)$ is a curved differential graded algebra as stated.
\end{proof}

Thus, in particular, if $A$ is a unital algebra and hence the curvature has the form \eqref{unital}, and if further $R=S$, then $(\Omega(A), d,\omega)$ is a curved differential graded algebra.

\begin{remark}\label{rem.deform}
Given an associative algebra $A$ with product $\mu$, and a linear map $\omega:A\ot A\to A$, for all $\lambda \in \KK$ one can consider deformation of the product
$$
\mu_{\omega,\lambda} := \mu + \lambda \omega.
$$
Following \cite{Ger:def}, the product $\mu_{\omega,\lambda}$ is associative up to the terms of order $\lambda$ or {\em infinitesimally associative} provided $\omega$ satisfies the cocycle condition, for all $a,b,c \in A$,
$$
a\omega(b\ot c) - \omega(ab\ot c) + \omega(a\ot bc) - \omega(a\ot b)c =0.
$$
If $(R,R,\omega)$ is a curved Rota-Baxter system on $A$ that satisfies assumptions of Proposition~\ref{prop.cdga}, i.e.\ the map $\omega$ is an $A$-bimodule homomorphism satisfying \eqref{ass}, then $\omega$ is a cocycle and hence the product $\mu_{\omega,\lambda}$ is infinitesimally associative.
\end{remark}

\section{Curved Rota-Baxter systems and pre-Lie algebras}
Recall from \cite{Ger:coh} that a vector space $A$ together with an operation $\circ : A\ot A\to A$ such that, for all $a,b,c\in A$,
\begin{equation}\label{pre}
(a\circ b - b\circ a)\circ c = a\circ (b\circ c) - b\circ (a\circ c),
\end{equation}
is called a {\em left pre-Lie algebra}.

\begin{proposition}\label{pre-Lie}
Let $(A,R,S,\omega)$ be a curved Rota-Baxter system. Then $A$ with operation $\circ$ defined by
$$
a\circ b = R(a)b - bS(a),
$$
is a left pre-Lie algebra if and only if, for all $a,b\in A$, 
$$
\omega(a\ot b) - \omega(b\ot a)
$$
is in the centre of $A$.
\end{proposition}

\begin{proof}
Starting with the left hand side of the defining equality for a pre-Lie algebra \eqref{pre}, we find, for all $a,b,c\in A$,
\begin{eqnarray*}
(a\circ b - b\circ a)\circ c &=& (R(a)b - bS(a) - R(b)a + aS(b))\circ c\\
&=& R(R(a)b)c - cS(R(a)b) - R(bS(a))c + cS(bS(a))\\
&&-R(R(b)a)c + cS(R(b)a) + R(aS(b))c - cS(aS(b))\\
&=& (R(a)R(b) - R(b)R(a)) c - c(S(a)S(b) -S(b)S(a))\\
&& + [c, \omega(a\ot b) - \omega(b\ot a)].
\end{eqnarray*}
On the other hand,
\begin{eqnarray*}
 a\circ (b\circ c) - b\circ (a\circ c) &=& R(a)(b\circ c) - (b\circ c) S(a) - R(b)(a\circ c) + (a\circ c) S(b)\\
 &=& R(a)R(b)c - R(a)cS(b) -R(b)c S(a) + cS(b)S(a)\\
 && - R(b)R(a) c + R(b)c S(a) + R(a)cS(b) - cS(a)S(b)\\
 &=& (R(a)R(b) - R(b)R(a)) c - c(S(a)S(b) -S(b)S(a)).
\end{eqnarray*}
Therefore, the operation $\circ$ makes $A$ into a pre-Lie algebra if and only if the commutator of $\omega(a\ot b) - \omega(b\ot a)$ with  all $c\in A$ vanishes, i.e.\ if and only if $
\omega(a\ot b) - \omega(b\ot a)
$ is a central element of $A$.
\end{proof}

Thus, in particular, if $\omega$ is a symmetric bilinear $A$-valued form on $A$, then $(A,\circ)$ is a left pre-Lie algebra.

We conclude this note by an example of a curved Rota-Baxter system that leads to a left pre-Lie algebra structure.

\begin{example} For an algebra $A$, let us take 
$$
r,s \in (A\ot A)^A := \{ x\in A\ot A \; |\; \forall a\in A, \; ax=xa\},
$$
write
$$
r= \sum r\su 1\ot r\su2, \qquad s = \sum s\su 1\ot s\su 2,
$$
and define an operation $\circ: A\ot A \to A$, by
\begin{equation}\label{circ}
a\circ b = \sum (r\su 1 a r\su2 - s\su 1a s\su 2)b.
\end{equation}
Then $(A,\circ)$ is a left pre-Lie algebra.
\end{example}

\begin{proof}
With the aid of $r,s\in (A\ot A)^A$ we can define linear maps $R,S: A\to A$ and $\omega :A\ot A\to A$ by
$$
R(a) = \sum r\su 1a r\su2, \qquad S(a) = \sum s\su 1a s\su2, \qquad  \omega(a\ot b) = -\sum r\su 1a r\su2s\su 1b s\su2.
$$
Notie that since $r,s \in (A\ot A)^A$, $R(a)$, $S(a)$ and $\omega(a\ot b)$ are all in the centre of $A$, and hence the formula \eqref{circ} can be written as
$$
a\circ b = R(a) b - bS(a).
$$
Writing $\sum \tilde{r}\su 1a \tilde{r}\su2$ for the second copy of $r$ and using the centrality of both $r$ and $s$, we can compute
\begin{eqnarray*}
R(a)R(b) - \omega(a\ot b) &=& \sum r\su 1a r\su2\tilde{r}\su 1b \tilde{r}\su2 + \sum r\su 1a r\su2s\su 1b s\su2\\
&=&  \sum \tilde{r}\su 1r\su 1a r\su2b \tilde{r}\su2 + \sum r\su 1a s\su 1b s\su2r\su2 \\
&=& R(R(a)b + aS(b)).
\end{eqnarray*}
By a similar computation one finds that also equation \eqref{s} is satisfied, and hence $(A,\circ)$ is a left pre-Lie algebra by Proposition~\ref{pre-Lie}.
\end{proof}


\begin{thebibliography}{99}{} 


\bibitem{Agu:inf}
M.\ Aguiar, {\em Infinitesimal Hopf algebras,} [in:]  {\em New trends in Hopf algebra theory (La Falda, 1999)}, Contemp.\ Math., {\bf 267}, Amer.\ Math.\ Soc., Providence, RI, (2000), pp.\ 1--29.

\bibitem{Agu:info}
M.\ Aguiar, {\em Infinitesimal bialgebras, pre-Lie and dendriform algebras}, [in:]  {\em Hopf Algebras}, Lecture Notes in Pure and Appl.\ Math., {\bf 237}, Dekker, New York (2004), pp.\ 1--33.
%
\bibitem{Bax:ana}  
G.\ Baxter, {\em An analytic problem whose solution follows from a simple algebraic identity}, Pacific J.\ Math.\ {\bf 10} (1960), 731--742. 

\bibitem{Brz:Rot} T.\ Brzezi\'nski, {\em Rota-Baxter systems, dendriform algebras and covariant bialgebras,} preprint arXiv:1503.05073 (2015).

%
\bibitem{Ebr:Lod} K.\ Ebrahimi-Fard, {\em Loday-type algebras and the Rota-Baxter relation,}  
Lett.\ Math.\ Phys.\ {\bf 61} (2002), 139--147.

\bibitem{EbrGuo:Rot} K.\ Ebrahimi-Fard \& L.\ Guo, {\em Rota-Baxter algebras in renormalization of perturbative quantum field theory,} [in:]  {\em Universality and renormalization},  Fields Inst.\ Commun.\ {\bf 50}, Amer.\ Math.\ Soc., Providence, RI, 2007, pp.\ 47--105.


\bibitem{Ger:coh} M.\ Gerstenhaber, {\em The cohomology structure of an associative ring}, Ann.\ Math.\ {\bf 78} (1963), 267--288.

\bibitem{Ger:def} M.\ Gerstenhaber, {\em  On the deformation of rings and algebras}, Ann.\ Math.\ {\bf 79} (1964), 59--103.

\bibitem{GetJon:alg} E.\ Getzler \& J.D.S.\ Jones, {\em $A_\infty$-algebras and the cyclic bar complex}, {Illinois J.\ Math.} {\bf 34} (1990), 256--283.

\bibitem{Guo:wha} L.\ Guo, {\em What is... a Rota-Baxter algebra?} Notices Amer.\ Math.\ Soc.\ {\bf 56} (2009), 1436--1437.

\bibitem{Lod:dia} J.-L.\ Loday, {\em Dialgebras}, [in:] {\em Dialgebras and Related Operads}, Lecture Notes in Math.\ {\bf 1763}, Springer, Berlin (2001) pp.\ 7--66.

\bibitem{PanVan:twi} F.\ Panaite \& F.\ Van Oystaeyen, {\em Twisted algebras, twisted bialgebras and Rota-Baxter operators,} arXiv:1502.05327 (2015).



\bibitem{Pos:non} L.\ Positselski, {\em Nonhomogeneous quadratic duality and curvature}, {Funct.\ Anal.\ Appl.} {\bf 27} (1993), 197--204.


\bibitem{Rot:Bax} G.-C.\ Rota, {\em  Baxter algebras and combinatorial identities.\ I, II}, 
Bull.\ Amer.\ Math.\ Soc.\ {\bf 75} (1969), 325--329, 330--334.



\bibitem{Sza:cla} B.\ Szablikowski, {\em Classical r-matrix like approach to Frobenius manifolds, WDVV equations and flat metrics}, J.\ Phys.\ A: Math.\ Theor., {\bf 48} (2015),  art.\ 315203.

\end{thebibliography}
\end{document}